\documentclass[12pt,reqno]{amsart}
\usepackage{amssymb}
\usepackage{mathrsfs}
\usepackage[all]{xy}

\theoremstyle{plain}
\newtheorem{prop}{Proposition}

\newtheorem{theo}[prop]{Theorem}

\theoremstyle{remark}
\newtheorem{rema}[prop]{Remark}
\theoremstyle{definition}
\newtheorem{defi}[prop]{Definition}

\numberwithin{equation}{section}

\newcommand{\A}{{\mathbb A}}

\newcommand{\PP}{{\mathbb P}}

\newcommand{\F}{{\mathbb F}}

\newcommand{\Z}{{\mathbb Z}}

\newcommand{\cO}{{\mathcal O}}
\newcommand{\cI}{{\mathcal I}}

\newcommand{\eqto}{\stackrel{\lower1.5pt\hbox{$\scriptstyle\sim\,$}}\to}
\newcommand{\eqdashto}{\stackrel{\lower1.5pt\hbox{$\scriptstyle\sim\,$}}\dashrightarrow}

\DeclareMathOperator{\Spec}{Spec}

\DeclareMathOperator{\Tor}{Tor}

\DeclareMathOperator{\Br}{Br}

\DeclareMathOperator{\Proj}{Proj}

\begin{document}
\title[Involution surface bundles]{Involution surface bundles over surfaces}
\author{Andrew Kresch}
\address{
  Institut f\"ur Mathematik,
  Universit\"at Z\"urich,
  Winterthurerstrasse 190,
  CH-8057 Z\"urich, Switzerland
}
\email{andrew.kresch@math.uzh.ch}
\author{Yuri Tschinkel}
\address{
  Courant Institute,
  251 Mercer Street,
  New York, NY 10012, USA
}

\email{tschinkel@cims.nyu.edu}

\address{Simons Foundation\\
160 Fifth Avenue\\
New York, NY 10010\\
USA}

\date{July 3, 2018} 

\begin{abstract}
We construct models of involution surface bundles over algebraic surfaces,
degenerating over normal crossing divisors,
and with controlled singularities of the total space.
\end{abstract}

\maketitle

\section{Introduction}
\label{sec:introduction}
In this paper,
we continue our investigation of del Pezzo fibrations,
initiated in \cite{KTsurf}.
Here we treat the case of fibrations in minimal del Pezzo surfaces
of degree $8$, also known as involution surfaces and described from
the arithmetic viewpoint, e.g., in \cite{auelbernardara}.
These include, as a special case, fibrations in quadric surfaces.

The main result of \cite{APS} gives a dictionary between quadric surface
bundles $X\to S$, where $S$ is regular and has dimension at most $2$,
degenerating along a regular divisor to
$\mathsf{A}_1$-singular quadrics, and sheaves of Azumaya quaternion algebras on
a double cover of $S$.
The approach taken there is based on Clifford algebras and orders in quaternion algebras.

Our results are more general in several respects:
\begin{itemize}
\item We allow more general degenerations, including
quadric surface fibrations where the defining equation drops rank by $2$ along a
regular divisor (Definition \ref{defn.mdisb}).
\item We allow the base $S$ to have arbitrary dimension (Theorem \ref{thm.mdisb}).
\item When $S$ is a surface over an algebraically closed field,
we allow degeneration along an arbitrary divisor and construct
(Theorem \ref{thm.main}) birational models with controlled singularities
degenerating over normal crossing divisors.
\end{itemize}

\begin{defi}
\label{defn.mdisb}
Let $S$ be a regular scheme, in which $2$ is invertible in the local rings.
An \emph{involution surface bundle} over $S$ is a flat projective morphism
$\pi\colon X\to S$ such that the locus $U\subset S$ over which $\pi$ is smooth is
dense in $S$ and the fibers of $\pi$ over points of $U$ are
involution surfaces.
An involution surface bundle $\pi\colon X\to S$ is said to have
\emph{mild degeneration} if every singular fiber is geometrically isomorphic to
one of the following reduced schemes:
\begin{itemize}
\item (Type I) a quadric surface with an $\mathsf{A}_1$-singularity;
\item (Type II) the self-product of a reduced singular conic;
\item (Type III) a union of two copies of the Hirzebruch surface $\F_2$, each with $(-2)$-curve
glued to a fiber of the other;
\item (Type IV) the product with $\PP^1$ of a reduced singular conic.
\end{itemize}
\end{defi}

\begin{theo}
\label{thm.main}
Let $k$ be an algebraically closed field of characteristic different from $2$,
$S$ a smooth projective surface over $k$, and
\[ \pi\colon X\to S \]
a morphism of projective varieties whose generic fiber is an involution surface.
Then there exists a commutative diagram
\[
\xymatrix{
\widetilde{X} \ar@{-->}[r]^{\varrho_X} \ar[d]_{\tilde\pi} & X \ar[d]^{\pi} \\
\widetilde{S} \ar[r]^{\varrho_S} & S
}
\]
such that
\begin{itemize}
\item $\varrho_S$ is a proper birational morphism,
\item $\varrho_X$ is a birational map that restricts to an
isomorphism over the generic point of $S$,
\item $\tilde\pi$ is projective and flat,
and the complement of the locus over which $\tilde\pi$ is smooth is a
simple normal crossing divisor $\widetilde{D}$, and
\item the restriction of $\tilde\pi$ over the complement in $\widetilde{S}$ of
the singular locus of $\widetilde{D}$ is a mildly degenerating
involution surface bundle.
\end{itemize}
\end{theo}

The singularities of $\widetilde{X}$ are analyzed in
Section \ref{sec:generic}, and
the structure of $\tilde\pi$ over the singular locus of $\widetilde{D}$
is analyzed in Section \ref{sec:codim2}.
Theorem \ref{thm.main} is restated, with this analysis, as
Theorem \ref{thm.mainst} in Section \ref{sec.models}.
Our second main result, Theorem \ref{thm.main2}, shows that
an involution surface bundle $\widetilde{X}$ as in
Theorem \ref{thm.main} is determined uniquely, up to
birational modification, by the data of a (possibly ramified)
double cover of $S$ and a $2$-torsion (possibly ramified)
Brauer group element on the double cover.
In the same setting,
we obtain in Theorem \ref{thm.main3} a strengthening of the dictionary of \cite{APS},
which allows degeneration over a simple normal crossing divisor.

Our approach is intrinsic to involution surfaces, and therefore the
good models that we produce, starting from a quadric surface fibration
with the more general degeneration pattern,
will not embed into a $\PP^3$-bundle over the base.
This more general degeneration pattern was an essential ingredient in
the application to stable rationality given in \cite{HPT}.

In our approach we continue the systematic use of root stacks,
as in \cite{HKTconic}, \cite{KTsurf}, \cite{HKTthreefolds}.

\medskip
\noindent
\textbf{Acknowledgments:}
We are grateful to Brendan Hassett and Alena Pirutka for stimulating discussions.
The second author is partially supported by NSF grant 1601912. 
Part of this work was done during a visit of the second author to the ETH, 
and he is grateful for ideal working conditions at FIM.

\section{Generic degenerations}
\label{sec:generic}
In this section we examine models of an involution surface over a DVR;
see \cite{corti} for a general discussion in the context of the
minimal model program.
We start with a description of an involution surface over a field $K$,
a surface that is geometrically isomorphic to $\PP^1\times\PP^1$.
Then we consider $K$ as the fraction field of a DVR and investigate the
possible degenerations.

Let $K$ be a field of characteristic different from $2$.
An involution surface $X$ over $K$ is classified by the data of an
\'etale $K$-algebra $L$ of degree $2$ and a central simple algebra
$B$ over $L$ of degree $2$; see, e.g., \cite[Exa.\ 3.3]{auelbernardara}.
The corresponding Brauer-Severi scheme is a conic $C$ over $L$, and
$X$ is isomorphic to the restriction of scalars of $C$ via the extension $L/K$.
We have $[B]=0$ in $\Br(L)$ if and only if $X(K)\ne \emptyset$
\cite[Lem.\ 7.1]{auelbernardara}.
Furthermore, $X$ has Picard number $2$ if and only if
$L\cong K\times K$, and $X$ is isomorphic to a quadric surface if and only if
$[B]$ corestricts to $0$ in $\Br(K)$, or equivalently, is in the image of the
restriction map $\Br(K)[2]\to \Br(L)[2]$.

The extension $L/K$ is called the \emph{discriminant extension} of
the involution surface $X$.
Rulings of $X$, i.e., projections to genus zero curves, are defined over $L$.

Now let $\mathfrak{o}_K$ be a DVR with fraction field $K$ and residue field $\kappa$,
also of characteristic different from $2$.
Let $\mathfrak{o}_L$ be the integral closure of $\mathfrak{o}_K$ in $L$.
Let $\beta\in \Br(L)[2]$ denote the Brauer class of $B$.
There are the following possibilities.

\begin{itemize}
\item $L$ is isomorphic to $K\times K$ or is a split unramified quadratic
extension of $K$.
Then $\Spec(\mathfrak{o}_L)$ has two closed points, each with residue field $\kappa$.
\begin{itemize}
\item $\beta$ is unramified at both closed points of $\Spec(\mathfrak{o}_L)$.
\item $\beta$ is unramified at one of the closed points points of $\Spec(\mathfrak{o}_L)$ and
ramified at the other.
\item $\beta$ is ramified at both closed points of $\Spec(\mathfrak{o}_L)$.
\end{itemize}
\item $L$ is an inert unramified quadratic extension field of $K$.
Then $\mathfrak{o}_L$ is a DVR whose residue field $\lambda$ is a quadratic extension of $\kappa$.
\begin{itemize}
\item $\beta$ is unramified.
\item $\beta$ is ramified.
\end{itemize}
\item $L$ is a ramified quadratic extension field of $K$.
Then $\mathfrak{o}_L$ is a DVR with residue field $\kappa$.
\begin{itemize}
\item $\beta$ is unramified.
\item $\beta$ is ramified.
\end{itemize}
\end{itemize}

In the first subcase of the first and second cases the description of
$X$ over $K$ as a restriction of scalars
extends to yield a smooth model of $X$ over $\mathfrak{o}_K$.
We organize the remaining cases as follows.

\begin{itemize}
\item I: $L$ is ramified, and $\beta$ is unramified.
\item II: $L$ is unramified, and $\beta$ is ramified at all closed points of $\Spec(\mathfrak{o}_L)$.
\end{itemize}

Types I and II occur for quadric bundles.
In this case, the rank of the defining quadratic equation drops by $1$, respectively by $2$.

\begin{itemize}
\item III: $L$ is ramified, and $\beta$ is ramified.
\item IV: $L$ is isomorphic to $K\times K$ or is split unramified, and
$\beta$ is ramified at one and unramified at the other closed point of $\Spec(\mathfrak{o}_L)$.
\end{itemize}

Types III and IV occur only in involution surfaces that are not
quadric surfaces.

\begin{rema}
\label{rem.relveryample}
For involution surfaces, as well as all of the singular degenerations in
Definition \ref{defn.mdisb}, the dualizing sheaf has dual which is
very ample with $9$-dimensional space of global sections.
As in \cite[Rem.\ 4.2]{KTsurf} and its consequences mentioned in
Section 7 of op.\ cit., this implies that any
mildly degenerating involution surface bundle $\pi\colon X\to S$ has
relatively very ample line bundle $\omega_{X/S}^\vee$, inducing an
embedding
\[ X\to \PP(\pi_*(\omega_{X/S}^\vee)^\vee) \]
into the projectivization of a rank $9$ vector bundle over $S$.
\end{rema}

If $S$ is irreducible, then an involution surface bundle
$\pi\colon X\to S$ has a well-defined discriminant extension, an \'etale
algebra extension of degree $2$ of the function field of the generic point.
We will describe the discriminant extension, as well, by the
integral closure of $S$ with respect to this \'etale algebra extension.
In general, $S$ is the disjoint union of its irreducible components,
and we extend the notion of discriminant extension accordingly.
If we fix, over the generic point $\eta$ of each component of $S$,
an \'etale $\kappa(\eta)$-algebra $\lambda(\eta)$ and
a compatible collection of identifications, for all
\'etale $\kappa(\eta)$-algebras, of identifications of rulings
of the base-change of $X$ with $\kappa(\eta)$-algebra homomorphisms to $\lambda(\eta)$,
then we will describe $X\to S$ as \emph{rigidified} by
$(\lambda(\eta)/\kappa(\eta))_\eta$ (or by the
integral closure of $S$ with respect to
$(\lambda(\eta)/\kappa(\eta))_\eta$).
Given a rigidification of $X\to S$, the Brauer class
$\beta$ at the generic point of each component of $T$ is determined
uniquely, and not only up to Galois conjugation.

\begin{defi}
\label{def.IVmarking}
Let $\mathfrak{o}_K$ and $\mathfrak{o}_L$ be as above, and let
$X$ be an involution surface over $K$, rigidified by $L/K$.
Suppose that $L$ is isomorphic to $K\times K$ or is split unramified over $K$.
We say that $X$ has \emph{Type IV marking} $\varepsilon\in \Spec(\mathfrak{o}_L)$
if $\varepsilon$ is a closed point of $\Spec(\mathfrak{o}_L)$, such that
the Brauer class $\beta\in \Br(L)[2]$ determined by the rigidified
involution surface extends to $\Spec(\mathfrak{o}_L)\smallsetminus \{\varepsilon\}$
but is ramified at $\varepsilon$.
\end{defi}

In the setting of an involution surface bundle $\pi\colon X\to S$, rigidified by $\psi\colon T\to S$,
we will speak of Type IV marking by a regular divisor $D'\subset T$,
mapping isomorphically to image $D\subset S$ disjoint from the
branch locus of $T\to S$, when
\begin{itemize}
\item $D$ is the locus in $X$ of Type IV fibers, and
\item for the generic point $\varepsilon$ of each component of $D'$, the restriction of $\pi$ over
$\Spec(\mathcal{O}_{S,\psi(\varepsilon)})$ has Type IV marking $\varepsilon$.
\end{itemize}

%

Let $\psi\colon T\to S$ be a finite flat morphism of degree $2$ between
regular schemes with branch locus $D_1\cup D_3$, and let
$D_2$ and $D_4$ be divisors in $S$.
Suppose that $D_1$, $\dots$, $D_4$ are regular and pairwise disjoint,
and $D'_4\subset T$ is a divisor mapping isomorphically by $\psi$ to $D_4$.
There are isomorphic copies of $D_1$ and $D_3$ in $T$, which we again
denote by $D_1$ and $D_3$.
We describe a procedure that transforms a
regular conic bundle $C\to T$ with singular fibers over
$\psi^{-1}(D_2)\cup D_3\cup D'_4$, to a
mildly degenerating involution surface bundle over $S$
with Type I fibers over $D_1$, Type II fibers over $D_2$,
Type III fibers over $D_3$, Type IV fibers over $D_4$, and Type IV marking $D'_4$.
We will see that the resulting involution surface bundle $X\to S$ has
singularties only over $D_2$, as in the following definition.

\begin{defi}
\label{defn.simple}
Let $S$ be a regular scheme, in which $2$ is invertible in the local rings.
We call a mildly degenerating involution surface bundle
$\pi\colon X\to S$ \emph{simple} if the locus where $\pi$ has
singular fibers is the disjoint union of four regular divisors
$D_1$, $\dots$, $D_4$, where $\pi$ has Type I fibers over $D_1$,
Type II fibers over $D_2$, Type III fibers over $D_3$, and Type IV fibers over $D_4$,
and if, letting $s\colon D_2\to X\times_SD_2$ denote the section which in each geometric fiber
is the intersection point of all the components, $X\smallsetminus s(D_2)$ is regular
and $X$ has ordinary double point singularities along $s(D_2)$.
\end{defi}

First, we consider the case that
$D_1$ and $D_3$ are empty:
\[
\xymatrix{
C \ar[r] & T\ar[d] \\
& S
}
\]
The vertical map is finite \'etale of degree $2$.
We restrict scalars along this to obtain $X$.

In the general case, we appeal to $[T/\mu_2]$, isomorphic to the root stack $\sqrt{(S,D_{13})}$ of
$S$ along $D_{13}:=D_1\cup D_3$, with finite \'etale cover by $T$.
The construction starts by replacing $C$ by a $\PP^1$-fibration
over the root stack $\sqrt{(T,D_3)}$, where the fibers over the
gerbe of the root stack $\mathcal{G}_3$ are $\PP^1$ with
nontrivial action of the stabilizer $\mu_2$.
This is accomplished by applying \cite[Prop.\ 3.1]{KTsurf} to
$C$ over the complement of the pre-image of
$D_{24}:=D_2\cup D_4$ and gluing with $C|_{T\smallsetminus D_{13}}$:
\[
P:=C|_{T\smallsetminus D_{13}}\cup
\big(\text{smooth $\PP^1$-fibration over
$\sqrt{(T\smallsetminus \psi^{-1}(D_{24}),D_3)}$}\big).
\]
The following diagram contains a fiber square whose vertical maps are
finite \'etale of degree $2$:
\begin{equation}
\begin{split}
\label{eqn.toomanyroots}
\xymatrix{
P \ar[r] &\sqrt{(T,D_3)}\ar[r]\ar[d]\ar@{}[dr]|{\square} & T\ar[d] \\
&\sqrt{([T/\mu_2],\mathcal{G}_3)} \ar[r] \ar[d]_{\substack{\text{iterated root:}\\ \{(D_1,2),(D_3,4)\}}}
& [T/\mu_2] \ar[dl]^{\text{root $D_{13}$}} \\
&S
}
\end{split}
\end{equation}
We restrict scalars to obtain a smooth $\PP^1\times\PP^1$-fibration
\[ \omega\colon
\Pi\to \sqrt{([T/\mu_2],\mathcal{G}_3)}. \]

It remains to modify $\Pi$ over $D_1$ and $D_3$ to obtain $\Pi'$ that descends to
an involution surface bundle over $S$.
Over the gerbe of the root stack $\mathcal{G}_1$ over $D_1$ the stabilizer $\mu_2$ acts on geometic fibers by
swapping the factors, fixing the diagonal.
We blow up the diagonal, to obtain in each geometric fiber a Hirzebruch surface $\F_2$.
Contracting the proper transform of $\omega^{-1}(\mathcal{G}_1)$ to a
copy of $\mathcal{G}_1$ yields $\Pi'$ away from $D_3$.
To see that $\Pi'$, away from $D_3$, descends to a
regular scheme, flat over $S$,
we invoke \cite[Prop.\ A.9]{KTsurf}, which identifies
the conormal sheaf $\cI_{\mathcal{G}_1/\Pi'}/\cI_{\mathcal{G}_1/\Pi'}^2$ with
the direct image of the conormal sheaf of the proper transform of
$\omega^{-1}(\mathcal{G}_1)$.
On geometric fibers this is
$H^0(\PP^1\times\PP^1,\cO_{\PP^1\times \PP^1}(1,1))$, which as
$\mu_2$-representation splits as trivial rank $3$ representation and
nontrivial rank $1$, the latter generated by a
local defining equation $t$ of $D_1$ in $T$ whose square is a
local defining equation of $D_1$ in $S$.
The modification over $D_3$ is more complicated and proceeds in two steps:
\begin{itemize}
\item[(i)] blow-up and contraction, descent to $[(T\smallsetminus D_1)/\mu_2]$;
\item[(ii)] further blow-up and contraction to obtain $\Pi'$.
\end{itemize}

Fix a point $x\in D_3$, and let
$A$ be the coordinate ring of an affine neighborhood of $x$ in $S$,
disjoint from $D_1$, $D_2$, and $D_4$.
Let $A':=A[t]/(t^2-f)$, where $f$ is a local defining equation of $D_3$.
As in the proof of loc.\ cit., there is an \'etale local model for $P$ of the form
\[
[\Proj(A''[u,v])/\mu_2],
\]
where $A''$ denotes $A'[s]/(s^2-t)$, with $\mu_2$ acting by scalar multiplication on
$s$ and $u$ and trivially on $v$.
Over the affine neighborhood, the fiber square in \eqref{eqn.toomanyroots} becomes
\[
\xymatrix{
[\Spec(A'')/\mu_2] \ar[r] \ar[d] & \Spec(A')\ar[d] \\
[\Spec(A'')/\mu_4] \ar[r] & [\Spec(A')/\mu_2]
}
\]
Now we claim that $\Pi$ is
\[ [\Spec(A'')\times \PP^1\times \PP^1/\mu_4], \]
where the action of $\mu_4$ on $\PP^1\times \PP^1$ is by
\[ i\cdot ((a:b),(c:d)):=((ic:d),(ia:b)). \]
Over each copy of $B\mu_4$ we find two geometric points
with $\mu_4$-stabilizer and one with $\mu_2$-stabilizer.

The claim reduces to the following computation, over
$\Spec(\Z[1/2])$ of the
restriction of scalars along $B\mu_2\to B\mu_4$
of $[\PP^1/\mu_2]$, where $\mu_2$ acts by
$(u:v)\mapsto (-u:v)$.
Consider a general
$\mu_4$-torsor $T\to S$ of schemes over $\Spec(\Z[1/2])$.
Then $\mu_2$-equivariant $T\to \PP^1$ morphisms are in bijective correspondence
with $\mu_4$-equivariant maps $T\to \PP^1\times \PP^1$,
for the $\mu_4$-action on $\PP^1\times \PP^1$ indicated above, by
\[ f\mapsto (f,(-i\cdot)\circ f\circ (i\cdot)), \]
where the expression on the right is independent of the choice of
primitive fourth root of unity $i$, by $\mu_2$-equivariance of
$f\colon T\to \PP^1$.
Restriction of scalars thus yields $[\PP^1\times \PP^1/\mu_4]$.

Step (i) is to blow up the locus with $\mu_4$-stabilizer, replacing each
geometric point with a copy of $\PP^2$, acted upon by $\mu_4$ with
one isolated fixed point, one pointwise fixed line, and generic stabilizer $\mu_2$.
The copy of $\PP^1\times \PP^1$ becomes a degree $6$ del Pezzo surface $F$, which
may be contracted to a line (as a conic bundle with, geometrically, two singular fibers).

For step (ii) we blow up the fixed locus for the $\mu_2$-action.
Geometric fibers then have four components:
two copies of $\F_2$ pointwise fixed by $\mu_2$
and two copies of $\F_1$.
The latter may be contracted to yield $\Pi'$ with contraction onto a locus
$Z$, whose fibers over $\mathcal{G}_3$ are reduced singular conics.


\begin{theo}
\label{thm.mdisb}
Let $S$ be a regular scheme, such that $2$ is invertible in the local rings
of $S$, and let $D\subset S$ be a regular divisor.
We suppose $D$ is written as a disjoint union of $D_1$, $D_2$, $D_3$, $D_4$,
and we are given a regular scheme $T$ with finite flat morphism
$\psi\colon T\to S$ of degree $2$, branched over $D_1\cup D_3$, and
divisor $D'_4$ mapping isomorphically to $D_4$.
Then the construction described above identifies, up to unique isomorphism:
\begin{itemize}
\item mildly degenerating regular conic bundles over $T$, with singular fibers
over $\psi^{-1}(D_2)\cup D_3\cup D'_4$, with
\item mildly degenerating simple involution surface bundles over $S$, rigidified by
$T\to S$, with singular fibers over $D_1$, $\dots$, $D_4$ as in Definition \ref{defn.simple}
and Type IV marking $D'_4$.
\end{itemize}
\end{theo}

\begin{rema}
\label{rem.DVRnew}
Even over a DVR the construction of Theorem \ref{thm.mdisb} includes new cases.
\end{rema}

\begin{proof}
We first treat that case that $T\to S$ is unramified, i.e., $D_1$ and $D_3$
are empty.
The restriction of scalars construction gives the
forwards construction, producing an involution surface bundle out of a
mildly degenerating regular conic bundle.
The description of the singularities is clear.
For the reverse construction, a mildly degenerating involution surface bundle
over $S$ determines an embedding in a $\PP^8$-bundle over $S$
(Remark \ref{rem.relveryample}).
Now the relative Grassmannian of planes meeting $X$ in a conic supplies,
after throwing away overdimensional components over $D_2$,
the reverse construction.
This produces a mildly degenerating conic bundle over $T$.
To see that this is regular, we suppose the contrary.
Then the forwards construction recovers $X$, but with singularities not
of the type that $X$ is assumed to possess.

For the general case,
we follow the strategy of the proof of \cite[Prop.\ 4.4]{KTsurf}.
The forwards construction has been described already.
We need to verify that the outcome of the modifications described over
$D_1$ and $D_3$ satisfy the claimed conditions.
This is carried out as in \cite[\S5]{KTsurf}.
We provide details only for the treatment over $D_3$.

For the singularity analysis of the stack obtained by
contracting $F$ to a line in step (i),
we pass to an affine \'etale neighborhood which trivializes the
discrimant extension, and let $H$, $I$, $J$, $K$ denote the proper transforms
of rulings of $X$, labelled so that the locus blown up in step (i)
consists of the intersection (of the images) of $H$ and $K$ and of
$I$ and $J$.
Then the following Cartier divisors determine Cartier divisors upon contraction:
\begin{gather*}
H+I,\qquad
J+K,\qquad
E_1+E_2+F, \\
2E_1+F+H+K,\qquad
2E_2+F+I+J.
\end{gather*}
There is thus a local defining equation
\[ yzs^2=uv \]
where $y$ and $z$ are sections of a line bundle
that has degree $1$ on each of the components of the
fiber $\PP^2\sqcup_{\PP^1}\PP^2$.
Thus, after the descent in (i), which replaces $s^2$ by $t$, we obtain geometric fibers over $\mathcal{G}_3$
with two points which are $\mathsf{A}_1$-singular points of the total space.
The $\mu_2$-action on $\PP^1$ is nontrivial, swapping the singular points.
The fixed locus for the action, in each of the geometric fibers,
consists of two disjoint lines, one in each component.

The singularity analysis in step (ii) proceeds as before away from
singular points of fibers.
At a singular point $r\in Z$, there is a new ingredient.
Passing to an \'etale neighborhood, we may assume that the
singular conic is reducible, with local functions $v$ and $w$,
each vanishing on one component and restricting to a local parameter
on the other.
Now the sequence, analogous to \cite[(5.2)]{KTsurf}, is only right exact:
the left-hand map, from the fiber of $\cI_Z/\cI_Z^2$ to the Zariski cotangent space
at $r$, has kernel of dimension one, generated by $vw$.
The conclusion is, however, the same:
the Zariski cotangent space at $r$ has dimension $n+3$, where $n$ is
the dimension of the local ring of $S$ at the image of $r$.

In fact, with the analysis of the reverse construction below,
it is possible to determine an \'etale local form of the
defining equation at $r$:
\[ u^2-v^2w^2=t^2. \]
Then, if we blow up
the locus $t=u=vw=0$, we find one chart with $t=ut_1$,
where $u=0$ defines the exceptional divisor, and defining equations
\begin{gather*}
vw = ux, \\
1-x^2 = t_1^2.
\end{gather*}
The fiber over $t=0$ consists of the exceptional divisor (two components)
and the proper transforms of components
\[ t_1=0,\,\,u=vw,\,\,x=1
\qquad\text{and}\qquad
t_1=0,\,\,u=-vw,\,\,x=-1, \]
and two points
\[ t_1=\pm 1,\,\,u=v=w=x=0, \]
at which the total space has $\mathsf{A}_1$-singularities.

The reverse construction needs to be checked over $D_1$ and $D_3$.
Let $r$ be a singular point of a Type I fiber of $\pi$.
With $D_1$ locally defined by $f=0$, we claim that there is an
\'etale local equation for $X$ at $r$ of the form
\[ x_1^2+x_2^2+x_3^2=f, \]
in affine $3$-space over the base as in \cite[\S 6]{KTsurf}.
The relative singular locus is, locally, a copy of $D_1$.
Denoting its ideal sheaf by $\cI$, we have $\cI/\cI^2$ locally free of
rank $3$, and we have an exact sequence as in \cite[(5.2)]{KTsurf}
with a fiber of $\cI/\cI^2$ at $r$
(a $3$-dimensional vector space over the residue field),
the Zariski cotangent space of $X$ at $r$ in the middle (dimension $n+2$),
and the Zariski cotangent space of $D_1$ at $r$ (dimension $n-1)$ on the right;
here $n$ denotes the dimension of the local ring of $S$ at the image $s$ of $r$,
with maximal ideal generated by $f$ and some $n-1$ further elements
$g_1$, $\dots$, $g_{n-1}$.
As in \cite[\S 5]{KTsurf},
the Zariski cotangent space of $X$ at $r$ surjects onto the Zariski cotangent space
of $\pi^{-1}(s)$ at $r$, with kernel spanned by
$g_1$, $\dots$, $g_{n-1}$ and $f$ mapping to zero in the
Zariski cotangent space of $X$ at $r$.
So $f$ maps to zero in every fiber of $\cI/\cI^2$, and hence
\begin{equation}
\label{eqn.fquadratic}
f\in \cI^2.
\end{equation}
Furthermore, the fiber of $\cI/\cI^2$ at $s$ maps isomorphically to the
Zariski cotangent space of $\pi^{-1}(s)$ at $r$, so we can let
$x_1$, $x_2$, $x_3$ be elements that locally generate $\cI/\cI^2$
and, as in loc.\ cit., use these elements to write $X$, \'etale locally,
as a hypersurface in affine $3$-space over the base.
Now \eqref{eqn.fquadratic} supplies a local defining equation
\[ f=\sum_{i,j=1}^3 a_{ij}x_ix_j, \]
where $(a_{ij})$ is a symmetric matrix with entries in the
coordinate ring of (an \'etale neighborhood of) the base.
The matrix is invertible since it has full rank over all points of the base.
This gives the desired \'etale local defining equation.
It is now straightforward to verify the reverse construction over $D_1$.

The reverse construction over $D_3$ is straightforward once we verify an
\'etale local equation of the form
\[ u^2-v^2w^2=f, \]
where $f$ is a local defining equation of $D_3$ at a closed point $s\in S$
and $u\pm vw$ are local defining equations of the components of the
fiber over $D_3$.
We start by passing to an \'etale affine neighborhood of $s$ in $S$
where the components of the fiber over $D_3$ are defined, as are the
components of the relative singular locus $Z$.
We let $r$ denote the singular point of $Z$ in the
fiber over $s$ and let
$u'$ and $u''$ be local defining equations at $r$ of the components
of the fiber over $D_3$, with
\begin{equation}
\label{eqn.uprimeudoubleprime}
f=u'u''.
\end{equation}
As well, we take $v$ to be a local function,
vanishing on one component of $Z$ and
restricting on the other component to
a local defining equation for the singular locus of $Z$ (a copy of $D_3$),
and we take $w$ to be an analogous local function, where the roles
of the two components of $Z$ are exchanged.
Letting $\cI$ denote the ideal sheaf of $Z$,
and denoting the maximal ideals corresponding to $s$ and $r$ by
$\mathfrak{m}$ and $\mathfrak{n}$, respectively,
we have the following right exact sequences of vector spaces over the
residue field $\kappa$ at $r$ (which is also the reside field at $s$):
\begin{gather}
(\cI/\cI^2)\otimes \kappa\to \mathfrak{n}/\mathfrak{n}^2\to
\mathfrak{n}/(\cI+\mathfrak{n}^2)\to 0, \label{eqn.D3exa1} \\
\mathfrak{m}/\mathfrak{m}^2\to \mathfrak{n}/\mathfrak{n}^2\to
\mathfrak{n}/(\mathfrak{m}\cO_{X,r}+\mathfrak{n}^2)\to 0. \label{eqn.D3exa3}
\end{gather}
By abuse of notation, $\cO_{X,r}$ denotes the structure sheaf at $r$
of $X$, pulled back over the \'etale neighborhood of $S$.
We let $n$ denote the dimension of $\cO_{S,s}$, with
$\mathfrak{m}$ generated by $f$ and additional local functions
$g_1$, $\dots$, $g_{n-1}$.

The Zariski cotangent space $\mathfrak{n}/\mathfrak{n}^2$ has dimension
$n+2$.
By \eqref{eqn.uprimeudoubleprime}, $f$ maps to $0$ in $\mathfrak{n}/\mathfrak{n}^2$
in \eqref{eqn.D3exa3}, where the space on the right,
the Zariski cotangent space of the fiber over $r$, has basis
$u'$, $v$, $w$;
alternatively, $u''$, $v$, $w$ is a basis.
It follows that $\mathfrak{n}/\mathfrak{n}^2$ has basis
\[ g_1,\dots,g_{n-1},u',v,w. \]
By \eqref{eqn.D3exa1}, noting that $Z$ is locally the
complete intersection defined by $u'$ and $u''$,
we have $\cI/\cI^2$ locally free of rank $2$, and
the images of generators $u'$ and $u''$ in
$\mathfrak{n}/\mathfrak{n}^2$ are equal up to scale.
Adjusting $u''$ (and $f$) by units, we may suppose that
$u'$ and $u''$ are equal in $\mathfrak{n}/\mathfrak{n}^2$.
It follows that
\[ u'-u''\in \cI\cap \mathfrak{n}^2. \]

Now we claim that
\[ u'\mathfrak{n}+(vw)=\cI\cap \mathfrak{n}^2. \]
Since the left-hand side is evidently contained in the right-hand side, it
suffices by Nakayama's lemma to show that $u'\mathfrak{n}$ and $vw$ span
\begin{equation}
\label{eqn.stuffmodn}
(\cI\cap \mathfrak{n}^2)/\mathfrak{n}(\cI\cap \mathfrak{n}^2).
\end{equation}
We have an exact sequence
\begin{equation}
\label{eqn.Imeetnsquared}
0\to \cI\cap \mathfrak{n}^2\to \cI\to (\cI+\mathfrak{n}^2)/\mathfrak{n}^2\to 0,
\end{equation}
where the space on the right is isomorphic to $\kappa$.
Tensoring with $\kappa$ yields a long exact sequence of $\Tor$ spaces.
It is readily verified that the map
$\Tor_1(\cI,\kappa)\to \Tor_1((\cI+\mathfrak{n}^2)/\mathfrak{n}^2,\kappa)$
is zero,
and the connecting homomorphism from
\[ \Tor_1((\cI+\mathfrak{n}^2)/\mathfrak{n}^2,\kappa)\cong
\Tor_1(\kappa,\kappa)\cong \mathfrak{n}/\mathfrak{n}^2 \]
is given by multiplication by $u'$.
It follows that the space \eqref{eqn.stuffmodn} has dimension $n+3$, and the
image of the connecting homomorphism has basis
\begin{equation}
\label{eqn.uprimen}
u'g_1,\dots,u'g_{n-1},u'^2,u'v,u'w.
\end{equation}
Now it suffices to show that this image does not contain
the class of $vw$.
The space
\[ \mathfrak{n}^2/(\mathfrak{m}\cO_{X,r}+\mathfrak{n}^3) \]
has dimension $5$, basis $v^2$, $w^2$, $u'v$, $u'w$, $vw$, and there the
image of $u'\mathfrak{n}$ does not contain $vw$.
So the claim is established.

In the right exact sequence obtained from
\eqref{eqn.Imeetnsquared} by tensoring with $\kappa$ we have the class of
$u'-u''$ in the middle mapping to zero on the right, hence
$u'-u''$ as an element of \eqref{eqn.stuffmodn} is not in the
span of \eqref{eqn.uprimen}.
From the claim and this observation it follows that
\[ u'-u''=u'h+vwk \]
for some $h\in \mathfrak{n}$ and $k\in \cO_{X,r}^\times$.
With
\[ u:=2k^{-1}u''+vw, \]
we compute
\[ u^2-v^2w^2=4(1-h)k^{-2}f, \]
which (after adjusting $f$ by a unit) is the desired local equation.

We conclude by using the case of unramified degree $2$ cover,
treated at the beginning of the proof, to associate to the
involution surface bundle over $\sqrt{([T/\mu_2],\mathcal{G}_3)}$,
which no longer has fibers of Type I or Type III,
a regular conic bundle over $\sqrt{(T,D_3)}$
with smooth fibers over the complement of $\psi^{-1}(D_{24})$.
By applying \cite[Prop.\ 3.1]{KTsurf} over the complement of
$\psi^{-1}(D_{24})$, this is converted to a regular conic bundle
over $T$.
\end{proof}

The following result gives an alternative description of the involution surface
in a simple setting.

\begin{prop}
\label{prop.easyD2}
Let $S$ be a regular scheme, such that $2$ is invertible in the
local rings of $S$, and let $D=D_2$ be a regular principal divisor, defined by
the vanishing of a regular function $f$ on $S$.
Then the construction of Theorem \ref{thm.mdisb},
applied to the degree $2$ cover $S\sqcup S\to S$ and conic bundle
$C_0\sqcup C_0$, where
\[ C_0:\,\,\,\, z_0z_2-fz_1^2=0\qquad\text{in}\qquad S\times \PP^2, \]
yields an involution surface bundle
\[ X\subset S\times \PP^8 \]
which sits in a commutative diagram
\[
\xymatrix{
[S'\times \PP^1\times \PP^1/\mu_2] \ar@{-->}[r] \ar@{-->}[d] & [S'/\mu_2]\times_SX \ar[d] \\
C_0\times_SC_0 \ar[r]_{\mathrm{Segre}} & X
}
\]
where $S'=\Spec(\cO_S[t]/(t^2-f))$ with $\mu_2$-action on $S'\times \PP^1\times \PP^1$ by
\[ (t,u:v,u':v')\mapsto (-t,-u:v,-u':v'), \]
the top map is the rational map sending
$(t,u:v,u':v')$ to
\[
(t^2u^2u'^2\!:\!tu^2u'v'\!:\!t^2u^2v'^2\!:\!tuvu'^2\!:\!uvu'v'\!:\!tuvv'^2\!:\!
t^2v^2u'^2\!:\!tv^2u'v'\!:\!t^2v^2v'^2),
\]
the left-hand map is the rational map given by
\[ (t,u:v,u':v')\mapsto (tu^2:uv:tv^2,tu'^2:u'v':tv'^2), \]
the right-hand map is projection,
and the bottom map is induced by the Segre embedding $\PP^2\times \PP^2\to \PP^8$.
\end{prop}

\begin{proof}
The involution surface bundle is obtained from the regular conic bundle
$C:=C_0\sqcup C_0$ by restriction of scalars along $S\sqcup S\to S$.
Since $C$ is pulled back from $C_0$ over $S$, this yields $C_0\times_SC_0$.
The rest is clear.
\end{proof}

\section{Codimension 2 phenomena}
\label{sec:codim2}
In this section we exhibit involution surface bundles
$X\to \A^2$, such that
$X\times_{\A^2}(\A^2\smallsetminus \{0\})$ is a mildly degenerating simple involution surface bundle
with $D_i$ equal to the $x$-axis and $D_j$, to the $y$-axis, for some $i$ and $j$
(where $D_i$ taken to be the union of the two axes when $i=j$).
These will be used in the following section to exhibit good models
of general involution surface bundles over surfaces, over an
algebraically closed field of characteristic different from $2$.

Let $k$ be a field of characteristic different from $2$.
We let $S=\A^2_k$, with coordinates $x$ and $y$.
We take $i\in \{1,2,3,4\}$ and $j=2$.
When $i$ is odd, we take $T=\A^2$, with coordinates
$s$ and $y$, with $x=s^2$, and when
$i$ is even we take $T$ to be the disjoint union of two copies of $S$.
We will also consider the case $i=j=4$, with $T$ the disjoint union of
two copies of $S$, which will split into two subcases according to the
form of the Type IV marking.

The strategy, when the machinery of Theorem \ref{thm.mdisb} does not
directly lead to a construction of $X\to S$, will be to
write down a model $\rho\colon P'\to S'$ of
$\pi\colon X^\circ\to S^\circ:=S\smallsetminus \{s_0\}$
over some
$S'$, isomorphic to $\A^2$,
with a finite flat covering $\tau\colon S'\to S$ and $\tau^{-1}(0)=\{s_0\}$.
Here, $X^\circ\to S^\circ$ is determined from a given cover $T^\circ\to S^\circ$ and
regular conic bundle $C^\circ\to T^\circ$ by Theorem \ref{thm.mdisb},
and $P'$ is supplied with a birational map to $S'\times_SX^\circ$,
which restricts to an isomorphism over the generic point of $S'$.
This determines $\pi_*(\omega_{X^\circ/S^\circ}^\vee)|_{P'}\to L$ for some
line bundle $L$ on $P'$.
By abuse of notation, we write the symbol for restriction
to $S'$ (resp., to a scheme over $S'$)
to denote pullback to $\tau^{-1}(S^\circ)$, followed by direct image in
$S'$ (resp., and further pullback);
the latter operation sends locally free coherent sheaves to
locally free coherent sheaves since
$\tau^{-1}(S^\circ)=S'\smallsetminus\{s_0\}$.
In our case $\rho$ will be smooth and
the isomorphism type of $L$ will be determined by its
restriction over the generic point of $S'$, from which we will have
$L\cong \omega_{P'/S'}^\vee$.
By adjunction, we will obtain
\begin{equation}
\label{eqn.rhostar}
\pi_*(\omega_{X^\circ/S^\circ}^\vee)|_{S'}\to \rho_*\omega_{P'/S'}^\vee.
\end{equation}
We will compute the image of \eqref{eqn.rhostar} and then, with linear algebra,
a collection of polynomial equations vanishing on the image of
\[ P' \dashrightarrow \PP(\pi_*(\omega_{X^\circ/S^\circ}^\vee)^\vee|_{S'}). \]
In every case these descend to yield equations of a subscheme $X$ of
projective space, flat over $S$.
Having $X$ flat over $S$, containing, and
over the generic point of $S$ agreeing with, the anticanonically embedded
$X^\circ$, it follows that
$X$ is the closure of $X^\circ$ under its
anticanonical embedding.
It will emerge, in every case, that $X$ is normal and
locally a complete intersection; as a consequence,
$\omega_{X/S}^\vee$ is locally free and defines the anticanonical embedding.

\subsection{I meets II}
\label{ss.ImeetsII}
We start with $S=\Spec(k[x,y])$ with $D_1:x=0$, $D_2:y=0$, degree $2$ cover $T=\Spec(k[s,y])$ with
$x=s^2$, and conic bundle $C\to T$, pulled back from $C_0\to S$, where
$C_0$ is defined by
\[ z_0z_2=yz_1^2 \]
in $S\times \PP^2$.
The recipe of Theorem \ref{thm.mdisb} starts with the restriction of
scalars of $C$ along
\[ T\to [T/\mu_2], \]
where $\mu_2$ acts by $(s,y)\mapsto (-s,y)$.
This gives
\[ [C_0\times_TC_0/\mu_2], \]
where $\mu_2$ acts on $T$ as before and by swapping the two factors $C_0$.
Then, we have to blow-up, contract, and descend over $D_1$.

Away from $D_1$, Proposition \ref{prop.easyD2} is applicable,
at least if we forget about $\mu_2$-equivariance.
We obtain a commutative diagram
\begin{equation}
\begin{split}
\label{eqn.ImeetsIIdiagram}
\xymatrix{
[T'\times \PP^1\times \PP^1/\mu_2] \ar@{-->}[r] \ar@{-->}[d] & [T'/\mu_2]\times \PP^8 \ar[d] \\
C_0\times_TC_0 \ar[r] & T\times \PP^8
}
\end{split}
\end{equation}
Here $T'$ is the cover of $T$, gotten by adjoining $t$ with $t^2=y$,
and $\mu_2$ acts by $(s,t)\mapsto (s,-t)$ and on each $\PP^1$-factor by
$(u:v)\mapsto (-u:v)$.
We take
\[ P':=T'\times \PP^1\times \PP^1 \]
with
\[ P'\dashrightarrow [C_0\times_TC_0/\mu_2]\eqdashto X^\circ. \]
Over the complement of $D_1$, diagram
\eqref{eqn.ImeetsIIdiagram} supplies equations for the image of
$\pi_*(\omega_{X^\circ/S^\circ}^\vee)|_{T'}\to H^0(P',\cO_{P'}(2,2))$:
\begin{align*}
k[s,s^{-1},t]\langle &t^2u^2u'^2,tu^2u'v',t^2u^2v'^2,\\
&tuvu'^2,uvu'v',tuvv'^2, \\
&t^2v^2u'^2,tv^2u'v',t^2v^2v'^2\rangle.
\end{align*}

Over the complement of $t=0$ the left-hand map in \eqref{eqn.ImeetsIIdiagram}
is an isomorphism.
We therefore can carry out the recipe of Theorem \ref{thm.mdisb} with $P'$.
With standard coordinates $uv$, $uv'$, $u'v$, $u'v'$ on
$\PP^1\times \PP^1$,
we perform a linear change of coordinates so that the defining equation
$uv'-u'v=0$ of the diagonal becomes a coordinate hyperplane.
The required blow-up is readily computed, and the contraction to the image
in $\PP^8$ is given by the following space of global sections of the
dual of the relative dualizing sheaf:
\begin{align*}
k[s,t,&t^{-1}]\langle s^2u^2u'^2,s^2(uv'+u'v)^2,(uv'-u'v)^2,s^2v^2v'^2,\\
&suu'(uv'-u'v),svv'(uv'-u'v),s(u^2v'^2-u'^2v^2),\\
&s^2uu'(uv'+u'v),s^2vv'(uv'+vu')\rangle
\end{align*}
Over all of $T'$ the space of global sections is given by the intersection:
\begin{align*}
k[s,t]\langle &s^2uvu'v',
stuu'(uv'-u'v),
stvv'(uv'-u'v),\\
&s^2tuu'(uv'+u'v),s^2tvv'(uv'+u'v),
t^2(uv'-u'v)^2,\\
&st^2(u^2v'^2-u'^2v^2),s^2t^2u^2u'^2,s^2t^2v^2v'^2\rangle.
\end{align*}

With linear algebra we find
\[ \frac{9\cdot 10}{2}-h^0(\PP^1\times \PP^1,(\omega_{\PP^1\times \PP^1}^\vee)^2)=45-25=20 \]
homogeneous polynomials of degree $2$
in $x_0$, $\dots$, $x_8$ with coefficients
in $k[s^2,t^2]$, of degree at most $1$ in $s^2$ and degree at most $1$
in $t^2$, vanishing on the
image of the map to $\A^2\times \PP^8$ given by the basis elements listed above.
These polynomials $f_1$, $\dots$, $f_{20}$ are are listed in Table \ref{ImeetsII}.
We introduce $\tilde x_1:=x_1+x_2$ and $\tilde x_5:=x_5+x_7+x_8$ and
verify just as in \cite[Lem.\ 9.1]{KTsurf} that
$k[s,t,x_0,\dots,x_8]/(f_1,\dots,f_{20})$ is free as a module over
$k[s,t,x_0,\tilde x_1,\tilde x_5]$ with basis
$1$, $x_2$, $x_3$, $x_4$, $x_6$, $x_7$, $x_8$, $x_3x_7$
for any field $k$ of characteristic different from $2$.
The fiber over $s=t=0$ is the union of a nonreduced quadric surface
and two quadric cones, and
\[ X_{\mathrm{I,II}}:=X \]
is smooth outside of a curve of $\mathsf{A}_1$-singularities over $D_2$ at
the point $(1:0:0:0:0:0:0:0:0)$.

\begin{table}
\[
\begin{array}{ccc}
x_0x_5-x_1x_2 &&                          s^2x_2^2+4x_0x_8-x_4^2 \\
x_0x_6-x_1x_4 &&                       x_2x_6-x_4x_5 \\
t^2x_0x_2-x_1x_8 &&                          4t^2x_0x_2+s^2x_2x_5-x_4x_6 \\
x_0x_6-x_2x_3 &&                          t^2x_0x_3-x_4x_7 \\
t^2x_0x_1-x_2x_7 &&                       t^2x_1^2-x_5x_7 \\
s^2x_1^2+4x_0x_7-x_3^2 &&                 t^2x_2^2-x_5x_8 \\
4t^2x_0^2+s^2x_1x_2-x_3x_4 &&             4t^2x_0x_5+s^2x_5^2-x_6^2 \\
x_1x_6-x_3x_5 &&                          t^2x_1x_3-x_6x_7 \\
4t^2x_0x_1+s^2x_1x_5-x_3x_6 &&            t^2x_2x_4-x_6x_8 \\
t^2x_0x_4-x_3x_8 &&                       s^2t^2x_0x_5-t^2x_3x_4+4x_7x_8
\end{array}
\]
\caption{Equations for the case of Type I meeting Type II.}
\label{ImeetsII}
\end{table}

\subsection{II meets II}
\label{ss.IImeetsII}
Here we have $D_2$ defined by $xy=0$ on $S=\Spec(k[x,y])$ and
$T=S\sqcup S$.
So we take $C_0$ to be the regular conic bundle on $S$ defined by
\[ xz_0^2+yz_1^2-z_2^2=0 \]
and, defining $C$ to be $C_0\times_ST$, obtain $X\to S$ as the
restriction of scalars of $C$ along $T\to S$:
\[ X_{\mathrm{II,II}}:=C_0\times_SC_0. \]
This has $\mathsf{A}_1$-singularities along four curves:
\begin{gather*}
x=z_1=z_2=z'_1=z'_2=0,\qquad\qquad
y=z_0=z_2=z'_0=z'_2=0,\\
x=y=z_2=z'_2=0,\,\,[z_0:z_1]=[\pm z'_0:z'_1].
\end{gather*}
There are two points of intersection of the curves,
where there is a more complicated toric sigularity:
\begin{align*}
&y(z_1+z'_1)(z_1-z'_1)=(z_2+z'_2)(z_2-z'_2),&&\text{on the chart $z_0=z'_0=1$}, \\
&x(z_0+z'_0)(z_0-z'_0)=(z_2+z'_2)(z_2-z'_2),&&\text{on the chart $z_1=z'_1=1$}.
\end{align*}

\subsection{III meets II}
\label{ss.IIImeetsII}
Here we assume that $k$ contains a primitive fourth root of unity $i$.
We have $S=\Spec(k[x,y])$ with $D_2:y=0$, $D_3:x=0$, degree $2$ cover
$T=\Spec(k[s,y])$ with $x=s^2$, and conic bundle
$C\to T$ defined by
\[ sz_0^2+yz_1^2-z_2^2=0 \]
Now the recipe of Theorem \ref{thm.mdisb} starts with an
application of \cite[Prop.\ 3.1]{KTsurf} to pass from $C\to T$ to
$P\to \sqrt{(T,D_3)}$.
Explicitly this is given by the rational map
\[ (rz_0+z_2:z_1:-rz_0+r_2), \]
where $r^2=s$,
to $[U\times \PP^2/\mu_2]$, where $U=\Spec(k[r,y])$ and
$\mu_2$ acts by $(r,y)\mapsto (-r,y)$ and by permutation of coordinates
on $\PP^2$.
With $w_0$, $w_1$, $w_2$ as homogeneous coordinates, $P$ is defined by
\[ w_0w_2=yw_1^2. \]
As in \S \ref{ss.ImeetsII}, Proposition \ref{prop.easyD2} is applicable,
leading to a commutative diagram
\begin{equation}
\begin{split}
\label{eqn.IIImeetsIIdiagram}
\xymatrix{
[U'\times \PP^1\times \PP^1/\mu_2] \ar@{-->}[r] \ar@{-->}[d] & [U'/\mu_2]\times \PP^8 \ar[d] \\
P\times_UP \ar[r] & U\times \PP^8
}
\end{split}
\end{equation}
where $U'=\Spec(k[r,t])$, with $t^2=y$ and $\mu_2$ acting on $U'$
by $(r,t)\mapsto (r,-t)$ and by $(u:v)\mapsto (-u:v)$ on each $\PP^1$ factor.

As in \S \ref{ss.ImeetsII} we obtain equations for global sections of the
dual of the relative dualizing sheaf:
\begin{align*}
k[r,r^{-1},t]\langle &t^2u^2u'^2,tu^2u'v',t^2u^2v'^2,\\
&tuvu'^2,uvu'v',tuvv'^2, \\
&t^2v^2u'^2,tv^2u'v',t^2v^2v'^2\rangle.
\end{align*}
Introducing
\[
\tilde u:=u+v,\qquad
\tilde v:=u-v,\qquad
\tilde u':=iu'+v',\qquad
\tilde v':=u'+iv',
\]
this becomes
\begin{align*}
k[&r,r^{-1},t]\langle t^2\tilde u\tilde v\tilde u'\tilde v',t\tilde u\tilde v(\tilde u'^2+\tilde v'^2),t^2\tilde u\tilde v(-\tilde u'^2+\tilde v'^2), \\
&t(-\tilde u^2+\tilde v^2)\tilde u'\tilde v',
(-\tilde u^2+\tilde v^2)(\tilde u'^2+\tilde v'^2),
t(-\tilde u^2+\tilde v^2)(-\tilde u'^2+\tilde v'^2), \\
&t^2(\tilde u^2+\tilde v^2)\tilde u'\tilde v',
t(\tilde u^2+\tilde v^2)(\tilde u'^2+\tilde v'^2),
t^2(\tilde u^2+\tilde v^2)(-\tilde u'^2+\tilde v'^2)\rangle.
\end{align*}
With respect to the new coordinates, the action of $\mu_2$ on
$\PP^1\times \PP^1$ in \eqref{eqn.IIImeetsIIdiagram} is by
\[ (\tilde u:\tilde v,\,\tilde u':\tilde v')\mapsto (\tilde v:\tilde u,\,-\tilde v':\tilde u'). \]

Over the complement of $D_2$, the left-hand map in
\eqref{eqn.IIImeetsIIdiagram} is an isomorphism, induced by the isomorphism
of smooth $\PP^1$-fibrations that is obtained by restricting
\[ [U'\times \PP^1/\mu_2\times \mu_2]\dashrightarrow P,\qquad
(r,t,u:v)\mapsto (r,t^2,tu^2:uv:tv^2), \]
over the complement of $D_2$; on the left, the first factor $\mu_2$ acts by
$(r,t)\mapsto (-r,t)$ on $U'$ and by $(u:v)\mapsto (v:u)$ on $\PP^1$, and the second,
by $(r,t)\mapsto (r,-t)$ on $U'$ and by $(u:v)\mapsto (-u:v)$ on $\PP^1$.
We carry out the construction of Theorem \ref{thm.mdisb} over the
complement of $D_2$ using this isomorphism.
We also use the coordinates $\tilde u$, $\tilde v$ introduced above,
with respect to which the first and second factors $\mu_2$ act by
\[
(r,t,\tilde u:\tilde v)\mapsto (-r,t,-\tilde u:\tilde v) \qquad\text{and}\qquad
(r,t,\tilde u:\tilde v)\mapsto (r,-t,\tilde v:\tilde u),
\]
respectively.

The computation of the restriction of scalars of $[U'\times \PP^1/\mu_2\times \mu_2]$
along $[U'/\mu_2\times \mu_2]\to [U'/\mu_4\times \mu_2]$ reduces to
the computation of the restriction of scalars of $[\PP^1/\mu_2\times \mu_2]$
along $B(\mu_2\times \mu_2)\to B(\mu_4\times \mu_2)$.
An analogous computation was performed in Section \ref{sec:generic}; we just state the result as
\[ [U'\times \PP^1\times \PP^1/\mu_4\times \mu_2], \]
with respective actions
\[
(\tilde u:\tilde v,\,\tilde u':\tilde v')\mapsto (i\tilde u':\tilde v',\,i\tilde u:\tilde v),\qquad
(\tilde u:\tilde v,\,\tilde u':\tilde v')\mapsto (\tilde v:\tilde u,\,-\tilde v':\tilde u').
\]

We carry out the blow-up and descent steps of
Theorem \ref{thm.mdisb} over $D_3$, following
\cite[Sect.\ 8]{KTsurf}.
First, we blow up $(0,0)$ and $(\infty,\infty)$ over $D_3$,
with respective exceptional divisors $E_1$ and $E_2$,
and semiample line bundle $\cO(2,2)(-2E_1-2E_2)$, leading to the
first contraction
\[ \varphi\colon B\ell_{\{(0,\ne 0,0,0)\}\cup \{(0,\ne 0,\infty,\infty)\}}(\Pi)\to
\widetilde{\Pi} \]
and descent
\[
\xymatrix{
\widetilde{\Pi}\ar[r]\ar[d]_{\chi}\ar@{}[dr]|{\square} &
[(U'\smallsetminus \{t=0\})/\mu_4\times \mu_2] \ar[d] \\
\widetilde{\Pi}_0 \ar[r] &
[(T'\smallsetminus \{t=0\})/\mu_2\times \mu_2]
}
\]
For some line bundle $L$ we have $\varphi^*L\cong\cO(2,2)(-2E_1-2E_2)$,
and there is a line bundle $L_0$ on $\widetilde{\Pi}_0$ that pulls back to $L$.
The second blow-up, of a locus $W$ in $\widetilde{\Pi}_0$ consisting of
two lines in the fibers over $D_3$, yields exceptional divisors
$\widetilde{E}_1$ and $\widetilde{E}_2$.
Then $L_0(-\widetilde{E}_1-\widetilde{E}_2)$ is
semiample and leads to the second contraction.
Its sections correspond to sections of $L$, vanishing on
$\chi^{-1}(W)$.
These are readily computed and lead to
\begin{align*}
k[r,t,t^{-1}]\langle &r^4\tilde u^2\tilde u'^2,
r\tilde u\tilde u'(\tilde u\tilde v'+\tilde u'\tilde v),
r^3\tilde u\tilde u'(\tilde u\tilde v'-\tilde u'\tilde v), \\
&r^2(\tilde u\tilde v'+\tilde u'\tilde v)^2,
-\tilde u^2\tilde v'^2+\tilde v^2\tilde u'^2,
r^3\tilde v\tilde v'(\tilde u\tilde v'+\tilde u'\tilde v), \\
&r^2(\tilde u\tilde v'-\tilde u'\tilde v)^2,
r\tilde v\tilde v'(\tilde u\tilde v'-\tilde u'\tilde v),
r^4v^2v'^2 \rangle.
\end{align*}

As in \S \ref{ss.ImeetsII}, the next step is to compute the
intersection of the two modules of sections of the dual of the
relative dualizing sheaf:
\begin{align*}
k[r,t]\langle &t^2(-u^2v'^2+v^2u'^2),
rt(uu'(uv'+u'v)+vv'(uv'-u'v)),\\
&rt^2(uu'(uv'+u'v)-vv'(uv'-u'v)),
r^2t(u^2v'^2+v^2u'^2),\\
&r^2t^2uu'vv',
r^3t(uu'(uv'-u'v)-vv'(uv'+u'v)),\\
&r^3t^2(uu'(uv'-u'v)+vv'(uv'+u'v)),\\
&r^4(-u^2+v^2)(u'^2+v'^2),
r^4t(u^2u'^2+v^2v'^2) \rangle.
\end{align*}
Relations $f_1$, $\dots$, $f_{20}$ are listed in Table \ref{IIImeetsII}.
As before, we verify that
$k[r,t,x_0,\dots,x_8]/(f_1,\dots,f_{20})$ is free as a module over
$k[r,t,x_0,x_1,x_7]$ with basis
$1$, $x_2$, $x_3$, $x_4$, $x_5$, $x_6$, $x_8$, $x_4x_5$
for any field $k$ of characteristic different from $2$.
The fiber over $r=t=0$ is a nonreduced scheme whose underlying reduced scheme
is a quadric cone.
The singularities in
\[ X_{\mathrm{III,II}}:=X \]
occur at
$(0:0:0:0:0:0:0:1:0)$ over $D_2$ and are of type
$\mathsf{D}_{\infty}$, an isolated line singularity \cite{siersma} with the
local analytic defining equation
\[ xz_0^2+z_1^2+z_2^2+z_3^2=0. \]

\begin{table}
\[
\begin{array}{ccc}
4x_0x_4+t^2x_1^2-x_2^2        &&     r^4x_0x_1+x_3x_6+2x_4x_5 \\
x_0x_5-2x_1x_4+x_2x_3         &&     r^4x_0x_2+t^2x_3x_5+2x_4x_6 \\
x_0x_6+t^2x_1x_3-2x_2x_4      &&     r^4x_1^2-x_3x_8+2x_4x_7 \\
2x_0x_8+x_1x_6+x_2x_5         &&     r^4x_0x_3-r^4x_1x_2-t^2x_3x_7+2x_4x_8 \\
2r^4x_0^2-2t^2x_0x_7+t^2x_1x_5+x_2x_6&&x_3x_8+2x_4x_7-x_5^2\\
x_0x_7-x_1x_5-x_3^2           &&     r^4x_1x_2-4x_4x_8-x_5x_6 \\
x_0x_8+x_2x_5+2x_3x_4         &&     r^4x_2^2+4t^2x_4x_7-2t^2x_5^2+x_6^2 \\
x_1x_8+x_2x_7+x_3x_5          &&     r^4x_1x_3+x_5x_8+x_6x_7 \\
2r^4x_0x_1-t^2x_1x_7-x_2x_8+x_3x_6&& 4r^4x_1x_4-r^4x_2x_3-t^2x_5x_7-x_6x_8 \\
r^4x_0^2-t^2x_3^2+4x_4^2      &&     2r^4x_1x_5+r^4x_3^2-t^2x_7^2+x_8^2
\end{array}
\]
\caption{Equations for the case of Type III meeting Type II.}
\label{IIImeetsII}
\end{table}

\subsection{IV meets II}
\label{ss.IVmeetsII}
Here we have $D_2:y=0$ and $D_4:x=0$ on $S=\Spec(k[x,y])$ and
$T=S\sqcup S$.
We take $C$ to be the regular conic bundle
\[ xz_0^2+yz_1^2-z_2^2=0 \]
on the first copy of $S$ and
\[ z_0^2+yz_1^2-z_2^2=0 \]
on the second copy of $S$.
Restriction of scalars yields
\[ X_{\mathrm{IV,II}}:\,\,\,\,xz_0^2+yz_1^2-z_2^2=0,\qquad z'^2_0+yz'^2_1-z'^2_2=0, \]
with singularities of type $\mathsf{D}_\infty$ at $(0:1:0,0:1:0)$ over $D_2$.

\subsection{IV meets IV}
\label{ss.IVmeetsIV}
With $D_4$ defined by $xy=0$ in $S=\Spec(k[x,y])$,
we divide into subcases according to
whether or not, on $T=S\sqcup S$,
the (closures of the) Type IV marked divisors meet.
If they meet, the outcome is a
product
\[ X'_{\mathrm{IV,IV}}:=\PP^1\times C_0, \]
where $C_0$ denotes the regular conic bundle $xz_0^2+yz_1^2-z_2^2=0$ over $\A^2$.
If not, then we get the locus
\[ X''_{\mathrm{IV,IV}}:\,\,\,\,xz_0^2+z_1^2-z_2^2=0,\qquad z'^2_0+yz'^2_1-z'^2_2=0. \]
In both cases, the total space is smooth.

\section{Existence of good models}
\label{sec.models}
We employ the involution surface bundles constructed in the previous section
as local models of involution surface bundles over surfaces.
We work over an algebraically closed field of characteristic different from $2$.

\begin{defi}
\label{defn.simple2}
Let $k$ be an algebraically closed field of characteristic different from $2$, and let
$S$ be a smooth surface over $k$.
We call an involution surface bundle $\pi\colon X\to S$ \emph{simple} if the
locus where $\pi$ has singular fibers is a simple normal crossing divisor
\[ D=D_1\cup D_2\cup D_3\cup D_4 \]
such that
\begin{itemize}
\item over the complement in $S$ of the singular locus $D^{\mathrm{sing}}$ of $D$,
$\pi$ is a mildly degenerating simple involution surface bundle with
singular fibers over $D_1$, $\dots$, $D_4$ as in Definition \ref{defn.simple};
\item $D_1$ and $D_3$ are smooth, disjoint from each other, and
disjoint from $D_4$; and
\item letting $\cO^h_{S,s}$ denote the henselization of a local ring of $S$ and
writing $\widehat{X}$ for $X\times_S\Spec(\cO^h_{S,s})$,
we have at all $s\in D^{\mathrm{sing}}$ that
$\widehat{X}$ is isomorphic to one of
\[
\widehat{X}_{\mathrm{I,II}}, \quad
\widehat{X}_{\mathrm{II,II}}, \quad
\widehat{X}_{\mathrm{III,II}}, \quad
\widehat{X}_{\mathrm{IV,II}}, \quad
\widehat{X}'_{\mathrm{IV,IV}}, \quad
\widehat{X}''_{\mathrm{IV,IV}}.
\]
\end{itemize}
\end{defi}

\begin{theo}
\label{thm.mainst}
Let $k$ be an algebraically closed field of characteristic different from $2$,
$S$ a smooth projective surface over $k$, and
\[ \pi\colon X\to S \]
a morphism of projective varieties whose generic fiber is an involution surface.
Then there exists a commutative diagram
\[
\xymatrix{
\widetilde{X} \ar@{-->}[r]^{\varrho_X} \ar[d]_{\tilde\pi} & X \ar[d]^{\pi} \\
\widetilde{S} \ar[r]^{\varrho_S} & S
}
\]
such that
\begin{itemize}
\item $\varrho_S$ is a proper birational morphism,
\item $\varrho_X$ is a birational map that restricts to an
isomorphism over the generic point of $S$,
\item $\tilde\pi$ is a simple involution surface bundle.
\end{itemize}
\end{theo}

\begin{proof}
By embedded resolution of singularities, we may suppose that
$X\to S$ is rigidified by a double cover $T\to S$,
branched over a simple normal crossing divisor.
Blowing up the singular locus of the branch locus, we obtain a
smooth branch locus.
The conic, corresponding to $X\times_S\Spec(k(S))$,
extends to a standard conic bundle on some smooth projective $\widetilde{T}$
with birational morphism $\widetilde{T}\to T$ \cite{sarkisovconicbundle}.
We remark, as well, that such a model exists as well over
$\widetilde{T}'$ for any smooth projective $\widetilde{T}'$ with
birational morphism $\widetilde{T}'\to \widetilde{T}$.
It is straightforward to exhibit such $\widetilde{T}'$, fitting into a
commutative diagram of smooth projective surfaces
\[
\xymatrix{
\widetilde{T}'\ar[r] \ar[d] & T \ar[d] \\
\widetilde{S} \ar[r] & S
}
\]
where the bottom morphism is proper birational and $\widetilde{T}\to \widetilde{S}$ is
finite of degree $2$.
We conclude by applying Theorem \ref{thm.mdisb} over
the complement of the singular locus of the image in $\widetilde{S}$
of the degeneracy locus of the conic bundle
filling in over the missing points as in the proof of Theorem 1.2 of \cite{KTsurf}.
\end{proof}

\begin{rema}
\label{rem.qm}
Given a quadric surface fibration, similar machinery can be employed to
produce models, fibered in quadric surfaces, with controlled singularities.
These have, as we recall, $D_3=D_4=\emptyset$.
Over $D_2$ the models have, generically, geometrically reducible fibers
(unions of two planes), and the models have $\mathsf{A}_1$-singularities
along a double cover of $D_2$.
An explicit description of the singularities
can be extracted from \cite[\S 5]{HPT}.
\end{rema}

\begin{theo}
\label{thm.main2}
Let $k$ be an algebraically closed field of characteristic different from $2$,
$S$ a smooth projective surface,
$T$ a smooth projective surface with morphism $\psi\colon T\to S$ that is generically
finite of degree $2$ or the disjoint union of two copies of $S$ with $\psi\colon S\sqcup S\to S$,
and $\beta$ a possibly ramified $2$-torsion Brauer group element on $T$,
i.e., an element of $\Br(k(T))[2]$ or a pair of elements of
$\Br(k(S))[2]$.
Then there exists a diagram
\[
\xymatrix{
\widetilde{X} \ar[d]_{\tilde\pi} \\
\widetilde{S} \ar[r]^{\varrho_S} & S
}
\]
where $\tilde\pi$ is a simple involution surface bundle over a smooth projective surface $\widetilde{S}$ and
$\varrho_S$ is a birational morphism,
such that the involution surface
\[ \widetilde{X}\times_{\widetilde{S}}\Spec(k(S)) \]
is classified by the \'etale $k(S)$-algebra $k(T)$, respectively $k(S)\times k(S)$, and
a central simple algebra with Brauer class $\beta$.
\end{theo}

The classification data of an involution surface over a field has been recalled
at the beginning of Section \ref{sec:generic}.

\begin{proof}
The function field of a surface over $k$ is a $C_2$-field, hence
$\beta$ is the class of a quaternion algebra.
\end{proof}

\begin{theo}
\label{thm.main3}
Let $k$ be an algebraically closed field of characteristic different from $2$,
$\psi\colon T\to S$ a finite morphism of degree $2$ of smooth surfaces over $k$ or a morphism of
the form $S\sqcup S\to S$ for a smooth surface $S$.
Let the branch divisor of $\psi$ be written as a disjoint union
\[ D_1\cup D_3 \]
of divisors on $S$, and let $D_2$ and $D_4$ be additional divisors on $S$ and $D'_4$ on $T$,
such that $D:=D_1\cup D_2\cup D_3\cup D_4$ is a simple normal crossing divisor on $S$,
for $i\ne j$ divisors $D_i$ and $D_j$ have no component in common,
$(D_1\cup D_3)\cap D_4=\emptyset$, and $\psi(D'_4)=D_4$, with every component of $D'_4$
mapping isomorphically to a component of $D_4$ and no two components mapping to the same
component of $D_4$.
The correspondence of Theorem \ref{thm.mdisb} over the complement of $D^{\mathrm{sing}}$
extends to a correspondence between
\begin{itemize}
\item regular conic bundles over $T$ with singular fibers over
$\psi^{-1}(D_2)\cup D_3\cup D'_4$, and
\item simple involution surface bundles over $S$,
rigidified by $T\to S$, where away from $D^{\mathrm{sing}}$ we have singular fibers over
$D_1$, $\dots$, $D_4$ as in Definition \ref{defn.simple} and
Type IV marking $D'_4$.
\end{itemize}
\end{theo}

\begin{proof}
We apply Theorem \ref{thm.mdisb} over the complement of $D^{\mathrm{sing}}$
with the argument as before for filling in over $D^{\mathrm{sing}}$.
\end{proof}

\bibliographystyle{plain}
\bibliography{pn}

\end{document}